\documentclass[a4paper]{amsart}
\usepackage{amsmath,amssymb,diagrams}
\usepackage[colorlinks]{hyperref}
\pdfoutput=1
\newcommand{\QQ}{\mathbb Q}
\renewcommand{\AA}{\mathbb A}
\newcommand{\ZZ}{\mathbb Z}
\newcommand{\NN}{\mathbb N}
\newcommand{\BB}{\mathbb B}
\newcommand{\calH}{\mathcal{H}}
\newcommand{\calL}{\mathcal{L}}
\newcommand{\calO}{\mathcal{O}}
\newcommand{\DD}{\mathbb D}
\newcommand{\kato}{\mathbf{z}_{\mathrm{Kato}}}
\newcommand{\vp}{\varphi}
\newcommand{\Qp}{\QQ_p}
\newcommand{\Zp}{\ZZ_p}
\newcommand{\Fp}{\mathbb{F}_p}
\newcommand{\Dcris}{\DD_{\mathrm{cris}}}
\newcommand{\Brig}{\BB_{\mathrm{rig},\Qp}^+}
\newcommand{\EA}{\calO_E\otimes\AA_{\Qp}^+}
\newcommand{\pe}{\varpi_E}
\newcommand{\Hiw}{H^1_{\mathrm{Iw}}}
\newcommand{\Nrig}{\NN_{\mathrm{rig}}}
\DeclareMathOperator{\uCol}{\underline{\mathrm{Col}}}
\DeclareMathOperator{\Gal}{Gal}
\DeclareMathOperator{\Fil}{Fil}
\DeclareMathOperator{\GL}{GL}
\DeclareMathOperator{\PS}{PS}
\DeclareMathOperator{\ord}{ord}

\renewcommand{\le}{\leqslant}
\renewcommand{\ge}{\geqslant}

\newtheorem{theorem}{Theorem}[section]
\newtheorem{proposition}[theorem]{Proposition}
\newtheorem{lemma}[theorem]{Lemma}
\newtheorem{corollary}[theorem]{Corollary}

\theoremstyle{definition}
\newtheorem{definition}[theorem]{Definition}

\theoremstyle{remark}
\newtheorem{remark}[theorem]{Remark}

\theoremstyle{plain}
\newtheorem{lettertheorem}{Theorem}

\begin{document}
\title{Wach modules and critical slope $p$-adic $L$-functions}

\author{David Loeffler}
\address{Warwick Mathematics Institute, University of Warwick, Coventry CV4 7AL, UK}
\email{D.A.Loeffler@warwick.ac.uk}

\author{Sarah Livia Zerbes}
\address{Department of Mathematics, University of Exeter, Exeter EX4 4QF, UK}
\email{S.Zerbes@exeter.ac.uk}
\thanks{The first author is supported by EPSRC Postdoctoral Fellowship EP/F04304X/2. The second author is supported by EPSRC Postdoctoral Fellowship EP/F043007/1.}

\begin{abstract}
We study Kato and Perrin-Riou's critical slope $p$-adic $L$-function attached to an ordinary modular form using the methods of \cite{leiloefflerzerbes10}. We show that it may be decomposed as a sum of two bounded measures multiplied by explicit distributions depending only on the local properties of the modular form at $p$. We use this decomposition to prove results on the zeros of the $p$-adic $L$-function, and we show that our results match the behaviour observed in examples calculated by Pollack and Stevens in \cite{pollackstevens08}.
\end{abstract}


\date{30th November 2010}

\maketitle

 \section{Introduction}
 
  \subsection{Background}
  
   Let $p \ge 3$ be prime, and let $f$ be a normalised, new modular eigenform of level $N$, character $\epsilon$ and weight $k \ge 2$, with $N$ prime to $p$. Then a classical construction of Amice--V\'elu and Vi\v{s}ik \cite{amicevelu75,vishik} gives rise to $p$-adic $L$-functions for $f$, which are distributions on $\mathbb{Z}_p^\times$ interpolating the critical values of the $L$-functions of $f$ and its twists by Dirichlet characters of $p$-power conductor.

   The construction depends on a choice of root of the Hecke polynomial of $f$ at $p$, and requires that the root be of ``non-critical slope'', i.e. that its $p$-adic valuation should be strictly less than $k-1$. The $L$-function corresponding to the root $\alpha$ is the unique distribution $L_{p, \alpha}$ on $\Zp^\times$ of order $h = \ord_p(\alpha)$ whose values at ``special'' characters of $\Zp^\times$, i.e.~those of the form $z \mapsto z^j \omega(z)$ where $0 \le j \le k-2$ and $\omega$ is a finite-order character, are given by
   \begin{equation}\label{eq:interpolating} \begin{gathered}\int_{\Zp^\times} z^j \omega(z)\, \mathrm{d}L_{p, \alpha}=  \\
    \begin{cases}
     \left(1 - p^j \alpha^{-1}\right)\left(1 - \epsilon(p) p^{k-2-j} \alpha^{-1} \right) \tilde L(f, 1, j + 1) &\text{if $n = 0$,}\\
     \alpha^{-n} p^{n(j+1)}\displaystyle\frac{\tilde L(f, \omega^{-1}, j + 1)}{G(\omega^{-1})} &\text{if $n \ge 1$,}
    \end{cases}\end{gathered}
   \end{equation}
   where $\tilde L(f, \omega^{-1}, j+1)$ is the complex $L$-value $L(f_{\omega^{-1}}, j+1)$ of the twisted form $f_{\omega^{-1}}$ divided by certain explicit transcendental factors (see equation \eqref{eq:Ltilde} below) and $G(\omega^{-1})$ is the Gauss sum.
   
   If the Hecke eigenvalue $a_p(f)$ is not a $p$-adic unit, then both roots have non-critical slope, and one obtains two $p$-adic $L$-functions, both of which are uniquely determined by the corresponding interpolation formula \eqref{eq:interpolating}. If $a_p(f)$ is a $p$-adic unit (the \emph{ordinary} case), then one root has non-critical slope (and in fact gives rise to a bounded measure) but the other does not, so one can only construct one $p$-adic $L$-function by these methods. Two constructions exist that redress the balance by constructing a ``critical slope $L$-function'' for ordinary eigenforms: a $p$-adic analytic approach via the theory of overconvergent modular symbols \cite{pollackstevens08}, and an algebraic approach via $p$-adic Hodge theory, using Kato's Euler system \cite{kato04}. Both approaches give a distribution of order $k-1$ on $\mathbb{Z}_p^\times$ with the same interpolation property at special characters, which depends on the restriction of the Galois representation of $f$ to a decomposition group at $p$. If the local representation is non-split, then the values of both of these critical-slope $L$-functions at special characters are given by \eqref{eq:interpolating}; if the local representation is split, the $L$-functions vanish at all such characters. However, these values do not uniquely determine a distribution of order $k-1$, and we cannot necessarily deduce that the $L$-functions arising from the two approaches are equal.

   In \cite[\S 9]{pollackstevens08}, Pollack and Stevens calculate the Newton polygon of the analytic critical-slope $L$-function $L^{\PS}_{p,\beta}$ in some explicit examples, and they observe that the distribution of the zeros follows interesting patterns which seem to be governed by the Iwasawa $\mu$ and $\lambda$-invariants of the unit-root $L$-function $L_{p,\alpha}$: when $p=3$ and $f$ is the twist of $X_0(11)$ by a quadratic character of conductor $D$ prime to $3$, for example, the numerical values suggest that the number of zeros inside the open disc of radius $r_n=\frac{1}{p^n(p-1)}$ is $p^n(p-1)+\lambda_D$, where $\lambda_D$ is the $\lambda$-invariant of $L^{\PS}_{p,\alpha}$. The same behaviour occurs when $p=5$ and the discriminant of the character is negative. However, when $p=5$ and the discriminant is positive, the number of zeros inside the open disc of radius $r_n$ is $p^n-1+\epsilon$, for some mysterious non-negative integer $\epsilon$ depending on $f$.

  \subsection{Statement of the main results}
  
   Let $p$ and $f$ be as above, and assume that $k \le p-1$ and that $f$ is ordinary. Let $\alpha$ and $\beta$ be the roots of the Hecke polynomial, and fix $\alpha$ to be the unit root. We choose a prime of the coefficient field of $f$ above $p$. In this introduction, let us assume for simplicity that the completion of the coefficient field at this prime is $\Qp$. Let $V_f^*$ be the dual of the $p$-adic representation attached to $f$, so it is a $2$-dimensional $\Qp$-vector space which is crystalline with Hodge-Tate weights  $0$ and $k-1$. Define $\Gamma=\Gal( \Qp(\mu_{p^\infty}) / \Qp)$, and write 
   \[\calL_{V^*_f} : \Hiw(\Qp,V^*_f)\rightarrow \calH(\Gamma)\otimes_{\Qp}\Dcris(V^*_f)\]
   for the Perrin-Riou regulator map. For any $z\in \Hiw(\Qp,V^*_f)$, we write $\calL_\alpha(z)$ (resp. $\calL_\beta(z)$) for the projection of $\calL_{V_f^*}(z)$ into the $\alpha$- (resp. $\beta$-)eigenspace of $\vp$. If $\kato \in  \Hiw(\Qp,V^*_f)$ is Kato's zeta element, then (for appropriate normalisations of the Frobenius eigenvectors) we have $L_\alpha(\kato) = L_{p,\alpha}$, and it is conjectured that $L_\beta(\kato)$ agrees with the critical slope $p$-adic $L$-function $L_{p, \beta}^{\PS}$ constructed by Pollack and Stevens (c.f. \cite[Remark 8.5]{pollackstevens09}). 
   
   To simplify the notation, write $L_{p,\beta}$ for $L_\beta(\kato)$. In this paper, we study $L_{p,\beta}$ using the description via Wach modules developed in \cite{leiloefflerzerbes10} and \cite{leiloefflerzerbes2}. This gives rise to a canonical subspace 
   \[ \left(\vp^*\NN(V^*_f)\right)^{\psi=0} \subseteq \calH(\Gamma)\otimes_{\Qp}\Dcris(V^*_f),\]
   stable under $\Gamma$ and of rank 2 as a $\Lambda(\Gamma)$-module, through which the map $\calL_{V_f^*}$ factors. In Section~\ref{basis}, we explicitly construct a basis $n_1,n_2$ of the Wach module $\NN(V_f^*)$. By comparing this basis with the $\vp$-eigenvector basis of $\Dcris(V_f^*)$, we obtain the following result:
   
   \begin{lettertheorem}
    If $V_f^*$ is not locally split, then there exist $L_{p,1},L_{p,2}\in \Lambda_{\Qp}(\Gamma)$ such that
    \[
     \begin{cases} 
      \alpha L_{p, \alpha} &= L_{p, 2}\\
      \beta L_{p, \beta} &= L_{p,1} \mathfrak{M}^{-1}\left( (1 + \pi)\vp\left(\frac{t}{\pi}\right)^{k-1} \right) - L_{p, 2}\mathfrak{M}^{-1}\left( (1 + \pi)\vp(a) \right)
     \end{cases}
    \]
    where, as an element of $\Qp[[t]]$, we have 
    \[ a = (k-2)!\alpha\left(\frac{1}{1 - p^{1-k} \mu} + (-1)^k \sum_{n \ge k} \binom{n-1}{k-2} \frac{B_n t^n}{n!(1 - \mu p^{n-k + 1})}\right).\]
   \end{lettertheorem}
   
   Since $L_{p,i}$ is bounded for $i=1,2$, the distribution of the zeros of $L_{p,\beta}$ is determined by the Newton polygons of $\left(\frac{t}{\pi}\right)^{k-1}$ and $a$. In Section \ref{newton}, we consider the case $k = 2$ and determine which of these two terms dominates, depending on the behaviour of the $\mu$-invariant of $L_{p,\alpha}$.
   
   \begin{lettertheorem}
   Let $\eta$ be a character of $\Delta$ and let $\lambda_1^\eta, \mu_1^\eta, \lambda_2^\eta, \mu_2^\eta$ be the Iwasawa $\lambda$- and $\mu$-invariants of the $\eta$-isotypical components of $L_{p, 1}$ and $L_{p, 2}$. Suppose that $V_f^*$ is non-split at $p$.
   \begin{enumerate}
    \item[(a)] If $\mu_2^\eta < \frac{1}{(p-1)^2} + \mu_1^\eta$, then for $n \gg 0$, $L_{p, \beta}^\eta$ has $p^n(p-1)^2$ zeros of valuation $\tfrac{1}{p^n(p-1)^2}$, and the total number of zeros of valuation $> r_n$ is $p^n(p-1) + \lambda_2^\eta$.
    \item[(b)] If $\mu_2^\eta > \frac{1}{p-1} + \mu_1^\eta$, then for $n \gg 0$, $L_{p, \beta}^\eta$ has $p^n(p-1)$ zeros of valuation $\tfrac{1}{p^n(p-1)^2}$, and the number of zeroes of valuation $> r_n$ is $p^n - 1 + \lambda_1^\eta$.
   \end{enumerate}
  \end{lettertheorem}
   
   Under the assumption that Kato's zeta element is integral, which is known in many cases, this explains the numerical phenomena observed by Pollack and Stevens (see the end of Section \ref{newton}).

  \subsection{Notation}
 
   As above, fix a prime $p\geq 3$, and let $\Gamma = \Gal(\Qp(\mu_{p^\infty}) / \Qp)$. Note that the cyclotomic character $\chi$ gives an isomorphism $\Gamma\cong\Zp^\times$. We write $\Gamma=\Delta\times\Gamma_1$, where $\Delta$ is cyclic of order $p-1$ and $\Gamma_1\cong\Zp$. We denote the absolute Galois group of $\Qp$ by $G_{\Qp}$. 
   
   We write $\Brig$ for the ring of power series $f(\pi)\in\QQ_p[[\pi]]$ such that $f(X)$ converges everywhere on the open unit $p$-adic disc. Equip $\Brig$ with actions of $\Gamma$ and a Frobenius operator $\vp$ by $g.\pi=(\pi+1)^{\chi(g)}-1$ and $\vp(\pi)=(\pi+1)^p-1$. We can then define a left inverse $\psi$ of $\vp$ satisfying
   \begin{equation}\label{psi}
    \vp\circ\psi(f(\pi))=\frac{1}{p}\sum_{\zeta^p=1}f(\zeta(1+\pi)-1).
   \end{equation}

   Inside $\Brig$, we have subrings $\AA_{\Qp}^+=\Zp[[\pi]]$ and $\BB_{\Qp}^+=\Qp\otimes_{\Zp}\AA_{\Qp}^+$. Moreover, the actions of $\vp$, $\psi$ and $\Gamma$ restrict to these rings. Finally, we write $t=\log(1+\pi)\in\Brig$ and $q=\vp(\pi)/\pi\in\AA_{\Qp}^+$. A formal power series calculation shows that $g(t) = \chi(g) t$ for $g \in \Gamma$ and $\vp(t) = pt$.

   \subsubsection{The Mellin transform}
   
    Given a finite extension $K$ of $\Qp$, denote by $\Lambda_{\calO_K}(\Gamma)$ (respectively $\Lambda_{\calO_K}(\Gamma_1)$) the Iwasawa algebra $\ZZ_p[[\Gamma]]\otimes_{\ZZ_p}\calO_K$ (respectively $\ZZ_p[[\Gamma_1]]\otimes_{\ZZ_p}\calO_K$) over $\calO_K$. We further write $\Lambda_K(\Gamma)=\QQ\otimes\Lambda_{\calO_K}(\Gamma)$ and $\Lambda_{K}(\Gamma_1)=\QQ\otimes\Lambda_{\calO_K}(\Gamma_1)$.  

    Let
    \[ \mathcal{H}=\{f\in\QQ_p[\Delta][[X]]:\text{$f$ converges everywhere on the open unit $p$-adic disc}\},\]
    and define $\calH(\Gamma)$ to be the set of $f(\gamma-1)$ with $f(X)\in\calH$. We may identify $\Lambda_{\Qp}(\Gamma)$ with the subring of $\calH(\Gamma)$ consisting of power series with bounded coefficients. Note that $\calH(\Gamma)$ may be identified with the continuous dual of the space of locally analytic functions on $\Gamma$, with multiplication corresponding to convolution, implying that its definition is independent of the choice of generator $\gamma$.

    The action of $\Gamma$ on $\Brig$ gives an isomorphism of $\calH(\Gamma)$ with $(\Brig)^{\psi=0}$, the Mellin transform
    \begin{align} 
     \mathfrak{M}: \calH(\Gamma) & \rightarrow (\Brig)^{\psi=0}\label{Mellin} \\
     f(\gamma-1) & \mapsto f(\gamma-1)(\pi+1).  \notag
    \end{align} 
    In particular, $\Lambda_{\Zp}(\Gamma)$ corresponds to $(\AA_{\Qp}^+)^{\psi=0}$ under $\mathfrak{M}$. Similarly, we define $\calH(\Gamma_1)$ as the subring of $\calH(\Gamma)$ defined by power series over $\Qp$, rather than $\Qp[\Delta]$. Then, $\calH(\Gamma_1)$ (respectively $\Lambda_{\Zp}(\Gamma_1)$) corresponds to $(1+\pi)\vp(\Brig)$ (respectively $(1+\pi)\vp(\AA_{\Qp}^+$)) under $\mathfrak{M}$.

   \subsubsection{Crystalline representations}\label{sect:wach}
   
    Let $E$ be a finite extension of $\Qp$, with ring of integers $\calO_E$. Fix a uniformizer $\pe$. For a crystalline $E$-linear representation $V$ of $G_{\Qp}$, denote its Dieudonn\'{e} module by $\Dcris(V)$. We say that $V$ is positive if its Hodge-Tate weights are $\leq 0$. The following result is shown in~\cite[\S II.1 and \S III.4]{berger04}: if $V$ is an $E$-linear representation, then $V$ is crystalline with Hodge-Tate weights in $[a,b]$ if and only if there exists a (necessarily unique) $E\otimes_{\QQ_p}\BB^+_{\QQ_p}$-module $\NN(V)$ contained in the $(\vp,\Gamma)$-module $\DD(V)$  of $V$ such that the following conditions are satisfied:
    \begin{enumerate}
     \item $\NN(V)$ is free of rank $d=\dim_E(V)$ over $E\otimes_{\QQ_p}\BB^+_{\QQ_p}$;
     \item the action of $\Gamma$ preserves $\NN(V)$ and is trivial on $\NN(V)/\pi\NN(V)$;
     \item $\vp(\pi^b\NN(V))\subset \pi^b\NN(V)$ and $\pi^b\NN(V)/ \vp^*(\pi^b\NN(V))$ is killed by $q^{b-a}$ where $q=\frac{\vp(\pi)}{\pi}$. (If $M$ is a $R$-module  equipped with a Frobenius $\vp$ where $R$ is any ring, then $\vp^*(M)$ denotes the $R$-module generated by $\vp(M)$.) 
    \end{enumerate}

    If $V$ is crystalline and positive, then we can take $b=0$ above, so $\vp$ preserves $\NN(V)$. In this case, if we endow $\NN(V)$ with the filtration $\Fil^i\NN(V)=\{x\in\NN(V)\mid \vp(x)\in q^i\NN(V)\}$, then $\NN(V)/\pi\NN(V)$ is a filtered $E$-linear $\vp$-module, and as shown in~\cite[\S III.4]{berger04} we have an isomorphism of filtered $\vp$-modules $\NN(V)/\pi\NN(V) \cong \Dcris(V)$. Also, as shown in \cite[II.2.1]{berger04}, $\Dcris(V)$ is contained in $\NN(V)\otimes_{\BB^+_{\Qp}}\Brig$, and in particular we can recover $\Dcris(V)$ as
    \[\Dcris(V)=\big(\NN(V)\otimes_{\BB^+_{\Qp}}\Brig\big)^{\Gamma}.\] 
    Moreover, we have a comparison isomorphism
    \begin{equation}\label{comparison} 
     \NN(V)\otimes_{\BB^+_{\Qp}}\Brig[t^{-1}]\cong \Dcris(V)\otimes_{\Qp}\Brig[t^{-1}],
    \end{equation}
    which holds independent of the Hodge-Tate weights of $V$. 
   
    If $T$ is a $G_{\QQ_p}$-stable lattice in $V$, then $\NN(T)=\NN(V)\cap\DD(T)$ is an $\calO_E\otimes_{\ZZ_p}\AA_{\QQ_p}^+$-lattice in $\NN(V)$, and by~\cite[\S III.4]{berger04} the functor $T\rightarrow \NN(T)$ gives a bijection between the $G_{\QQ_p}$-stable lattices $T$ in $V$ and the $\calO_E\otimes_{\ZZ_p}\AA_{\QQ_p}^+$-lattices in $\NN(V)$ satisfying the conditions (1)-(3) above.
   
    Finally, it is easy to see from the construction that for all $j\in\ZZ$, $\NN(V)$ and $\NN(V(j))$ are related by 
    \[ \NN(V(j))=\pi^{-j}\NN(V)\otimes e_j,\]
    where $e_j$ is a basis for $\Qp(j)$.

   \subsubsection{Iwasawa cohomology}\label{sect:iwasawa}

    If $V$ is a $p$-adic representation of $G_{\Qp}$, and $T$ is a $G_{\Qp}$-stable $\Zp$-lattice in $V$, we define the \emph{Iwasawa cohomology} of $T$ to be 
    \[ \Hiw(\Qp, T) := \varprojlim_n H^1(\Qp(\mu_{p^n}), T).\]
    As shown in~\cite{perrinriou94}, this is a $\Lambda_{\Zp}(\Gamma)$-module of finite rank. Define $\Hiw(\Qp, V) := \Qp \otimes_{\Zp} \Hiw(\Qp, T)$ for any $G_{\Qp}$-stable lattice $T$; this definition is independent of the choice of  $T$.

    By \cite[Theorem A.3]{berger03}, if the Hodge-Tate weights of $V$ are $\ge 0$ and $V$ has no quotient isomorphic to the trivial representation, there is a canonical isomorphism of $\Lambda_{\Qp}(\Gamma)$-modules
    \begin{equation}\label{Fontaineisom}
    \Hiw(\Qp, V) \cong \NN(V)^{\psi = 1},
    \end{equation}
    which also identifies $\Hiw(\Qp, T)$ with $\NN(T)^{\psi = 1}$ for each lattice $T$. These constructions clearly commute with the additional $\calO_E$-linear and $E$-linear structures when $V$ is an $E$-linear representation and $T$ is a $G_{\Qp}$-stable $\calO_E$-lattice.


 \section{Setup}

  Let $p$ be prime and let $f$ be a normalised new modular eigenform of level $N$, character $\epsilon$ and weight $k \ge 2$, with $N$ prime to $p$ and $k \le p-1$. Let $F$ be the coefficient field of $f$; we fix a choice of prime of $F$ above $p$ and let $E$ be the completion of $F$ at that prime, which we regard as a subfield of $\overline{\QQ}_p$.

  We assume that $f$ is \emph{ordinary}, i.e.~$a_p(f)$ is a $p$-adic unit. Thus the roots of the Hecke polynomial $X^2 - a_p(f) X + p^{k-1} \varepsilon$, where $\varepsilon = \epsilon(p)$, are elements of $E$ with valuations $0$ and $k - 1$. Let $\alpha$ be the unit root and $\beta = p^{k-1} \varepsilon \alpha^{-1}$ the non-unit root. We define $\mu := p^{k-1} \alpha / \beta = \varepsilon^{-1} \alpha^2$, which is a $p$-adic unit. By the Deligne--Ramanujan--Petersson bound on $|a_p|$, all embeddings of $\mu$ into $\mathbb{C}$ have complex absolute value $p^{k-1}$; in particular, $\mu \ne 1$.

  We define $(\rho_f, V_f)$ to be the ``cohomological'' $p$-adic representation attached to $f$, so the characteristic polynomial of \textit{geometric} Frobenius at primes $\ell \nmid Np$ is $X^2 - a_\ell X + \ell^{k-1} \epsilon(\ell)$. Then it is well known that
  \[ \rho_{f} \mid_{G_{\Qp}} \cong 
    \begin{pmatrix}
    \lambda(\alpha) & * \\
    0 & \chi^{1-k} \lambda(\varepsilon \alpha^{-1})
    \end{pmatrix}
  \]
  where $\lambda(x)$ denotes the unramified character of $G_{\Qp}$ mapping geometric Frobenius to $x$. In particular, $\rho_f \mid_{G_{\Qp}}$ has an unramified subrepresentation.

  Since $f$ has level prime to $p$, $\rho_{f} \mid_{G_{\Qp}}$ is crystalline, and the characteristic polynomial of $\vp$ on $\Dcris(\rho_f)$ is $X^2 - a_p X + p^{k-1} \varepsilon$. The unramified subrepresentation of $\rho_f \mid_{G_{\Qp}}$ corresponds to the $\vp = \alpha$ eigenspace.

  We shall mostly work with the ``homological'' representation $(\rho_f^*, V_f^*)$, the linear dual of $V_f$. This is crystalline at $p$ and the \emph{arithmetic} Frobenius at primes $\ell \nmid Np$ has the Hecke polynomial as its characteristic polynomial. On the decomposition group it is given by 
  \[ \rho_{f}^* \mid_{G_{\Qp}} \cong 
    \begin{pmatrix}
    \chi^{k-1} \lambda(\varepsilon^{-1} \alpha) & * \\
    0 & \lambda(\alpha^{-1}),
    \end{pmatrix}\]
  so the $\vp$-eigenvalues on $\Dcris(V_f^*)$ are $\beta^{-1} = p^{1-k} \varepsilon^{-1} \alpha$ (corresponding to the subrepresentation with Hodge-Tate weight $k - 1$) and $\alpha^{-1}$.

  We let $T_f$ and $T_f^*$ denote the canonical $G_{\QQ}$-stable $\calO_E$-lattices in $V_f$ and $V_f^*$ defined using the cohomology of the modular curve $X_1(N)_{\overline{\QQ}}$ with coefficients in $\mathbb{Z}_p$, as in \cite[\S 8.3]{kato04}.

 
 \section{A sequence of polynomials}

  In this section we define a certain sequence of polynomials which are related to the Eulerian polynomials.

  \begin{definition}
   Let the polynomials $h_j(X)$, for $j \ge 0$, be defined by $h_0(X) = 1$ and
   \begin{equation}\label{defhj}
    h_{j}(X) = (1 + X)\left(- X \frac{\mathrm{d}}{\mathrm{d}X} + j\right)h_{j-1}(X).
   \end{equation}
  \end{definition}
   
  It is immediate that $h_j$ is a monic polynomial of degree $j$ with integral coefficients and constant term $h_j(0) = j!$.

  \begin{proposition}\label{prop:bernoulli}
   For any $j \ge 0$, the following identity of formal power series holds in $\QQ[[t]]$:
   \[ \left(\frac{t}{e^t - 1}\right)^{j + 1}h_j(e^t - 1) = j! \left(1 + (-1)^j \sum_{n \ge j+1} \binom{n-1}{j} \frac{B_n t^n}{n!}\right),\] 
   where $B_n$ is the $n$-th Bernoulli number.
  \end{proposition}

  \begin{proof}
   The case $j = 0$ of the proposition is the definition of the Bernoulli numbers,
   \[ \frac{t}{e^t - 1} = 1 + \sum_{n \ge 1} \frac{B_n t^n}{n!}.\]
   If $D_j$ is the operator $\left(-t \tfrac{d}{dt} + j\right)$, then we compute that
   \[
    D_j \cdot \left(\frac{t}{e^t - 1}\right)^{j}h_{j-1}(e^t - 1) = \left(\frac{t}{e^t - 1}\right)^{j+1} h_j(e^t - 1)
   \]
   and
   \begin{multline*}
    D_j \cdot (j-1)!\left(1 + (-1)^{j-1} \sum_{n \ge j} \binom{n-1}{j-1} \frac{B_n t^n}{n!}\right)\\
       = j! \left( 1 + (-1)^j \sum_{n \ge j+1} \binom{n-1}{j} \frac{B_n t^n}{n!}\right).
   \end{multline*}
   So the proposition holds for all $j$ by induction.
  \end{proof}

 
 \section{Calculating the \texorpdfstring{$(\vp, \Gamma)$}{(phi, Gamma)}-module}\label{basis}

  Let $V_1 = E(\chi^{k-1} \lambda(\varepsilon^{-1} \alpha))$ and $V_2 = E(\lambda(\alpha^{-1}))$. As a consequence of \cite[Theorem 0.2(ii)]{colmez08}, the space
  \[ \operatorname{Ext}^1_{G_{\Qp}}(V_2, V_1) \cong H^1(\Qp, \chi^{k-1} \lambda(\mu)) \]
  is 1-dimensional, and thus up to isomorphism there are exactly two extensions: one split and one non-split. In this section we give an explicit description of the $(\vp, \Gamma)$-module corresponding to the non-split extension $V$.

  Since the $(\vp,\Gamma)$-module functor is exact, $\DD(V)$ is a free module of rank 2 over $\BB_{\Qp}$ with a basis $v_1, v_2$, where $v_1$ is a basis for the $(\vp, \Gamma)$-module of $V_1$ and the image of $v_2$ is a basis for the $(\vp, \Gamma)$-module of the quotient $V_2$. Thus in this basis $(v_1, v_2)$, $\vp$ acts via the matrix
  \[ P = \begin{pmatrix} \varepsilon^{-1} \alpha & x \\ 0 & \alpha^{-1}\end{pmatrix}\]
  and an element $\gamma \in \Gamma$ acts via
  \[ G = \begin{pmatrix} \chi(\gamma)^{k-1} & \alpha y \\ 0 & 1 \end{pmatrix}\]
  for some $x, y \in \BB_{\Qp}$. In the split case, we may clearly take $x = y = 0$; so let us assume we are in the non-split case. Since the actions of $\vp$ and $\Gamma$ commute, we must have $P \vp(G) = G \gamma(P)$, so $x$ and $y$ satisfy
  \[ (\mu \vp - 1)(y) = (\chi(\gamma)^{k-1} \gamma - 1)(x).\]
  This, of course, is exactly the requirement that $(x, y)$ forms a $1$-cocycle in the Herr complex (see e.g.~\cite[\S I.4]{cherbonniercolmez99}) calculating $H^1(\Qp, \chi^{k-1} \lambda(\mu))$.
  We shall construct an explicit choice of $(x, y)$ realising the non-split extension.

  \begin{proposition}
   If $x_k = \pi^{1-k} h_{k-2}(\pi)$, where $h_j$ are the polynomials of definition \ref{defhj}, then we have
   \[ (\chi(\gamma)^{k-1} \gamma - 1) x_k \in \AA^+_{\Qp}\]
   for all $\gamma \in \Gamma$; and $h_{k-2}$ is the unique monic polynomial of degree $k - 2$ such that this holds.
  \end{proposition}

  \begin{proof}
   Let $f$ be any element of $\Qp[[t]]$ with $f(0) = 1$, and let $F = \pi^{1-k} f$. Then we find that
   \[ (\chi(\gamma)^{k-1} \gamma - 1) F = t^{1-k}(\gamma - 1)(t^{k-1} F) = t^{1-k}(\gamma - 1)\left( (\tfrac{t}{\pi})^{k-1} f\right).\]
   Hence we have $(\chi(\gamma)^{k-1} \gamma - 1) F \in \Qp[[t]]$ if and only if $(\gamma - 1)\left( (\tfrac{t}{\pi})^{k-1} f\right) \in t^{k-1} \Qp[[t]]$, or equivalently $(\tfrac{t}{\pi})^{k-1} f \in 1 + t^{k-1} \Qp[[t]]$. By construction, $f = h_{k-2}(\pi) = h_{k-2}(e^t - 1)$ satisfies this (and it is obvious that there is a unique polynomial with this property up to scaling). Since $F$ lies in $\AA_{\Qp}$, we have  $(\chi(\gamma)^{k-1} \gamma - 1) F \in \AA_{\Qp} \cap \Qp[[t]] = \AA^+_{\Qp}$.
  \end{proof}

  Recall that $\mu := \varepsilon^{-1} \alpha^2 \ne 1$.

  \begin{lemma}\label{lemma:noncoboundary}
   There is no $z \in E \otimes_{\Qp} \BB_{\Qp}$ such that $(\mu\vp - 1)(z) = \pi^{1-k} h_{k-2}(\pi)$.
  \end{lemma}

  \begin{proof}
   We assume (for simplicity of notation) that $E = \Qp$. Observe that we have a decomposition $\AA_{\Qp} = \pi \AA^{+}_{\Qp} \oplus \AA^{\leqslant 0}_{\Qp}$, where $\AA^{\le 0}_{\Qp}$ consists of those series with only non-positive powers of $\pi$; thus $\AA^{\le 0}_{\Qp}$ is a subring isomorphic to $\Zp\langle X \rangle$, where $X = \pi^{-1}$. Also, $\vp$ preserves $\AA^{\le 0}_{\Qp}$, since the series expansion of $\vp(\pi^{-1})$ lies in $\AA^{\le 0}_{\Qp}$; indeed this gives an action of $\vp$ on $\Zp\langle X \rangle$ lifting the canonical Frobenius on $\Fp[X]$.

   If $z$ is such that $(\mu\vp - 1)(z) = \pi^{1-k} h_{k-2}(\pi)$, then we must have $z \in \BB^{\le 0}_{\Qp} = \AA^{\le 0}_{\Qp}[p^{-1}]$, and the constant term of $z$ is zero. If $z \ne \AA^{\le 0}_{\Qp}$, then there is some $j > 0$ such that $z' = p^j z$ is in $\AA^{\le 0}_{\Qp}$ and its mod $p$ reduction $\overline{z'}$ is a non-zero element of $X\Fp[X]$. But then $z'$ must satisfy $(\mu \vp - 1)(\overline{z'}) = 0$, which is clearly impossible as $\vp$ increases the degree of any non-constant polynomial.

   Hence $z \in \AA^{\le 0}_{\Qp}$. Then $\overline{z} \in \Fp[X]$ and $(\mu \vp - 1)(\overline{z})$ is a non-constant polynomial in $X$ of degree $k-1 < p$. It is clear that no such polynomial can lie in the image of $\mu \vp - 1$.
  \end{proof}

  We now define $x_k = \pi^{1-k} h_{k-2}(\pi) - \delta_k$, where $\delta_k = (-1)^k \frac{B_{k-1}}{(k-1)}$. Since $k-1 < p-1$, $\delta_k \in \Zp$, and hence $x_k \in \AA_{\Qp}$. Moreover, the choice of $\delta$ implies that $t^{k-1} x_k = (k-2)! + O(t^k)$ in $\Qp[[t]]$; hence $(\chi(\gamma)^{k-1} \gamma - 1) (x_k) \in \pi \AA^+_{\Qp}$.

  \begin{proposition}
   There exists $y_k \in \EA$ solving
   \[ (\mu \vp - 1)(y_k) = (\chi(\gamma)^{k-1} \gamma - 1) (x_k).\]
  \end{proposition}

  \begin{proof}
   It is clear that $\mu \vp - 1$ is surjective on $\pi \left(\EA\right)$ for any $\mu \in \calO_E$: if $z \in \pi \AA^+_{\Qp}$, then we have $\vp^n(z) \in \vp^n(\pi) \AA^+_{\Qp}$, and $\vp^n(\pi)$ tends to zero in the $(p, \pi)$-adic topology of $\AA^+_{\Qp}$. Thus the series expansion $-\sum_{n \ge 0} \mu^n \vp^n(z)$ converges, and its limit is clearly a preimage of $z$.
  \end{proof}

  Note that $y_k$ depends on $k$, $\mu$, and $\gamma$, but $x_k$ depends only on $k$. By construction, $(x_k, y_k)$ determines a class in $H^1$ of the Herr complex, and Lemma \ref{lemma:noncoboundary} shows that this class is not zero. Since $H^1(\Qp, \chi^{k-1} \lambda(\mu))$ is one-dimensional, we deduce that the $(\vp, \Gamma)$-module we have constructed corresponds to the unique non-split extension of the two factors. 


 \section{Calculations in Wach modules}
 
  We deduce that the $(\vp, \Gamma)$-module of the unique non-split extension $V$ has a basis $(v_1, v_2)$ for which the matrices of $\vp$ and the generator $\gamma$, in the basis $(v_1, v_2)$, are given by
  \[ P = \begin{pmatrix} \varepsilon^{-1} \alpha & x_k \\ 0 & \alpha^{-1}\end{pmatrix}\]
  and 
  \[ G = \begin{pmatrix} \chi(\gamma)^{k-1} & \alpha y \\ 0 & 1 \end{pmatrix}.\]
  We let $n_1 = \pi^{1-k}v_1$ and $n_2 = v_2$. Then the matrices of $\vp$ and $\gamma$ in the basis $(n_1, n_2)$ are given by
  \[ P' = \begin{pmatrix} \frac{\pi^{k-1}}{\vp(\pi)^{k-1}} \varepsilon^{-1} \alpha & \pi^{k-1} x_k \\ 0 & \alpha^{-1}\end{pmatrix}\]
  and 
  \[ G' = \begin{pmatrix} \frac{\pi^{k-1}}{\gamma(\pi^{k-1})} \chi(\gamma)^{k-1} & \pi^{k-1} \alpha y \\ 0 & 1\end{pmatrix}.\]

  It follows that if $N$ is the $E \otimes \BB^+_{\Qp}$-span of $n_1$ and $n_2$, the module $N$ satisfies the conditions (1)-(3) of \S \ref{sect:wach}, so $N = \NN(V)$, the Wach module of $V$. Moreover, the $\EA$-span of $(n_1, n_2)$ is a Wach module over $\EA$, and hence is $\NN(T)$ for a $G_{\Qp}$-stable $\calO_E$-lattice $T$ in $V$.

  \begin{lemma}\label{lemma:lattices}
   For each $c, d \in \ZZ$ with $c \le d$, there is a lattice $T_{c, d} \in V$ whose Wach module is the $\EA$-span of $\pe^c n_1, \pe^d n_2$, and every $G_{\Qp}$-stable $\calO_E$-lattice in $V$ is one of these. The residual representation $T_{c, d} / \pi_E T_{c, d}$ is non-split if $c = d$ and split otherwise.
  \end{lemma}

  \begin{proof}
   It is clear that $\pe^c n_1, \pe^d n_2$ span a Wach module for every $c \le d$, and hence correspond to a lattice in $V$. By construction the mod $\pe$ reduction of the cocycle $(x_k, y_k)$ is nontrivial, so $T_{0, 0} = T$ is residually non-split, and hence the same holds for $T_{c, c}$ for any $c$.

   Let $T'$ be any lattice in $V$; by scaling we may assume that it is contained in $T$, and not contained in $\pe T$. Then the image of $T'$ in $T / \pe T$ is a nontrivial Galois-stable subspace; so it is the subspace corresponding to the reduction of $n_1$, as $T / \pe T$ is non-split and thus has a unique 1-dimensional subspace. By devissage, we deduce that $T' = \pe^j T + T_1$ for some $j$, where $T_1 = T \cap V_1$ is the lattice corresponding to $n_1$. This corresponds to the cocycle $(\pe^j x_k, \pe^j y_k)$, and hence is residually split if $j \ge 1$.
  \end{proof}

  \begin{proposition}\label{prop:basis}
   The space $\left(\vp^* \NN(V)\right)^{\psi = 0}$ is free of rank $2$ as a $\Lambda_{E}(\Gamma)$-module, and a basis is given by $(1 + \pi) \vp(n_1), (1 + \pi)\vp(n_2)$. More specifically, if $T_{c,d}$ denotes the lattice in $V$ defined in Lemma \ref{lemma:lattices}, then $\pe^c (1 + \pi) \vp(n_1), \pe^d (1 + \pi) \vp(n_1)$ are a basis of $\big(\vp^*\NN(T_{c,d})\big)^{\psi=0}$ as a free $\Lambda_{\calO_E}(\Gamma)$-module.
  \end{proposition}

  \begin{proof}
   It is easy to check that the functor $(\vp^*(-))^{\psi = 0}$ is exact, so we have a short exact sequence of $\Lambda_{\calO_E}(\Gamma)$-modules
   \[ 0\rTo \big(\vp^*\NN(T_1)\big)^{\psi=0}\rTo \big(\vp^*\NN(T)\big)^{\psi=0}\rTo \big(\vp^*\NN(T_2)\big)^{\psi=0}\rTo 0.\]
   Now as shown in \cite[Theorem 3.5]{leiloefflerzerbes10}, $\big(\vp^*\NN(T_1)\big)^{\psi=0}$ is a free $\Lambda_{\calO_E}(\Gamma)$-module of rank $1$, and there exists a basis $n_1'$ of $\NN(T_1)$ which is congruent to $n_1 \bmod \pi$ such that $(1+\pi)\vp(n'_1)$ is a $\Lambda_{\calO_E}(\Gamma)$-basis of $\big(\vp^*\NN(T_1)\big)^{\psi=0}$. Observe that an element in $\calO_E\otimes\AA^+_{\Qp}$ is invertible if and only if its constant term is a unit in $\calO_E$. Write $n_1=a n_1'$ with $a\in \big(\AA^+_{\Qp}\big)^\times$. By Proposition 3.10 in {\em op.cit.}, $\vp(\pi)^i(1+\pi)\vp(n_1')\in (1-\gamma)\big(\vp^*\NN(T_1)\big)^{\psi=0}$. It follows that if we write $(1+\pi)\vp(n_1)=
   \alpha.(1+\pi)\vp(n_1')$, then $\alpha\in\Lambda_{\calO_E}(\Gamma)^\times$, so $(1+\pi)\vp(n_1)$ is also a basis of $\big(\vp^*\NN(T_1)\big)^{\psi=0}$. 
   
   Similarly, we can show that if $\bar{n}_2$ denotes the image of $n_2$ in $\NN(T_2)$, then $(1+\pi)\vp(\bar{n}_2)$ is a $\Lambda_{\calO_E}(\Gamma)$-basis of $\big(\vp^*\NN(T_2)\big)^{\psi=0}$. As $(1+\pi)\vp(n_2)$ is a lift of $(1+\pi)\vp(\bar{n}_2)$, this implies the result. 
  \end{proof}

  We now calculate $\Dcris(V)$ using Berger's comparison isomorphism~\eqref{comparison}. Since the highest Hodge-Tate weight is $k-1$, we have
  \[ 
   \Dcris(V) = \left( \left(\tfrac{t}{\pi}\right)^{1-k} \Brig \otimes_{\BB^+_{\Qp}} \NN(V) \right)^\Gamma.
  \]
  It is clear that $\left(\tfrac{t}{\pi}\right)^{1-k} n_1$ is $\Gamma$-stable, and that it is a $\vp$-eigenvector with eigenvalue $\beta^{-1} = p^{1-k} \varepsilon^{-1} \alpha$.

  To find a $\alpha^{-1}$ eigenvector lifting $n_2$ (which is clearly an eigenvector in $\Dcris(V_2)$) we must find $a, b \in \Brig$ such that $\left(\tfrac{t}{\pi}\right)^{1-k} a n_1 + b n_2$ is $\Gamma$-stable and killed by $\alpha \vp - 1$, with $b = 1$ modulo $\pi$. Comparing coefficients of $n_2$, we find that $b = 1$. Writing out the equations $\gamma(v) = v$ and $\vp(v) = \alpha^{-1}v$, where $v = \left(\tfrac{t}{\pi}\right)^{1-k} a n_1 + n_2$, we need to have
  \[ (1 - \gamma)(a) = t^{k-1} \alpha y\]
  and
  \[ (1 - p^{1-k}\mu \vp)(a)  = t^{k-1} \alpha x_k.\]
  The existence and uniqueness of a solution to these equations is a consequence of the fact that $V$ is known to be crystalline. If $a$ is the solution, then we have the following formulae for the eigenvectors in $\Dcris(V)$:

  \begin{proposition}\label{prop:evects} A basis of $\vp$-eigenvectors in $\Dcris(V)$ are given by:
   \begin{align*}
    v_{\beta^{-1}} &= -a(0) \left(\tfrac{t}{\pi}\right)^{1-k} n_1,\\
    v_{\alpha^{-1}} &= \left(\tfrac{t}{\pi}\right)^{1-k} a n_1 + n_2,
   \end{align*}
   where $a$ is the unique solution in $\Brig$ to $(1 - p^{1-k}\mu \vp)(a)  = t^{k-1} \alpha x_k$.
  \end{proposition}

  The motivation for the factor $-a(0)$ is the following lemma: 
  
  \begin{lemma}\label{lemma:filtn}
   The subspace $\Fil^{2-k} \Dcris(V) = \Fil^{0} \Dcris(V)$ of $\Dcris(V)$ is spanned by the vector $v_{\alpha^{-1}} + v_{\beta^{-1}} = n_2 + (a - a(0)) n_1$.
  \end{lemma}

  \begin{proof}
   We know that
   \[ \Fil^j \Dcris(V) = \{ v \in \Dcris(V) : \vp(v) \in q^j \Nrig(V)\}.
   \]
      
   Since $v_{\beta^{-1}} = -a(0) \left(\tfrac{t}{\pi}\right)^{1-k} n_1$, we clearly have $v_{\beta^{-1}} \ne \Fil^{2-k}\Dcris(V)$. Hence we can find some scalar $\lambda$ such that $\Fil^{2-k} \Dcris(V)$ is spanned by $v = v_{\alpha^{-1}} + \lambda v_{\beta^{-1}}$. We have
   \begin{align*}
    \vp(v) &= \alpha^{-1} v_{\alpha^{-1}} + \lambda \beta^{-1} v_{\beta^{-1}}\\
    &= \left(\tfrac{t}{\pi}\right)^{1-k}(a \alpha^{-1} - \lambda a(0) \beta^{-1}) n_1 + \alpha^{-1} n_2.
   \end{align*}
   This is in $q^{2-k} \Nrig(V)$ if and only if $(a \alpha^{-1} - \lambda a(0) \beta^{-1}) \in q \Brig$; in other words, if and only if $a(\zeta_p - 1) = p^{1-k} \lambda  \mu a(0)$ where $\zeta_p$ is a primitive $p$-th root of unity. Since $a = t^{k-1}\alpha x_k + p^{1-k} \mu \vp(a)$, and $t$ vanishes at $\zeta_p - 1$, we have $a(\zeta_p - 1) = p^{1-k} \mu \vp(a)(\zeta_p - 1) = p^{1-k} \mu a(0)$. Thus the unique solution is $\lambda = 1$, as claimed.
  \end{proof}

  Since the Dieudonn\'e module of a lattice $T \subseteq V$ is the image of $\NN(T)$ in $\Dcris(V) = \NN(V) / \pi \NN(V)$, we see that:

  \begin{corollary}\label{cor:latticecd}
   If $T_{c, d}$ is the lattice in $V$ defined above, and $\Fil^0 \Dcris(T_{c,d})$ is the $\calO_E$-span of $\pe^d (v_{\alpha^{-1}} + v_{\beta^{-1}})$.
  \end{corollary}

  Note that there is a well-defined map $\Brig \to \Qp[[t]]$, which is determined by sending $\pi$ to $e^t-1$ and whose image contained in the ring of power series converging for $|t| < p^{-1/(p-1)}$.

  \begin{proposition}\label{prop:a-formula}
   As elements of $\Qp[[t]]$, we have
   \[ \frac{a}{(k-2)! \alpha} = \frac{1}{1 - p^{1-k} \mu} + (-1)^k \sum_{n \ge k} \binom{n-1}{k-2} \frac{B_n t^n}{n!(1 - \mu p^{n-k + 1})}.\]
  \end{proposition}

  \begin{proof}
   By definition, we have $(1 - p^{1-k}\mu \vp)(a)  = \alpha t^{k-1}x_k$. Since 
   \[ x_k = \pi^{1-k} h_{k-2}(\pi) - (-1)^k B_{k-1} / (k-1),\]
   we have 
   \[ t^{k-1} x_k = \left(\frac{t}{e^t - 1}\right)^{k-1} h_{k-2}(e^t - 1) -  (-1)^k \frac{B_{k-1}t^{k-1}}{(k-1)}\]
   as elements of $\Qp[[t]]$. Substituting in the formula of Proposition \ref{prop:bernoulli}, the $t^{k-1}$ terms cancel, and we obtain
   \[ t^{k-1} x_k = (k-2)!\left(1 + (-1)^k \sum_{n \ge k}\binom{n-1}{k-2} \frac{B_n}{n!} t^n\right)\]
   Since $\vp(t^n) = p^n t^n$, we deduce that
   \[ \frac{a}{(k-2)! \alpha} = \frac{1}{1 - p^{1-k} \mu} + (-1)^k \sum_{n \ge k} \binom{n-1}{k-2} \frac{B_n t^n}{n!(1 - \mu p^{n-k + 1})},\]
   as claimed.
  \end{proof}

  We now consider the relation between this abstract representation $V$ and the representation $V_f^*$ attached to $f$. We have $V_f^* \cong V_{\bar f}(k-1)$, where $\bar f$ is the complex conjugate of $f$. We may identify $\DD_{\mathrm{dR}}(V_{\bar f})$ with the $\bar f$-isotypical component of the de Rham cohomology of $X_1(N)$ with coefficients in the standard line bundle $\omega^k$; in particular, $\bar f$ gives a canonical basis for the nontrivial filtration step. 

  \begin{proposition}
   If the representation $V_f^*$ is non-split as a $G_{\Qp}$-representation, there is a unique isomorphism of $G_{\Qp}$-representations 
   \[ V_f^* \cong V\]
   mapping $t^{1-k} \bar{f}$ to $v_{\alpha^{-1}} + v_{\beta^{-1}}$.
  \end{proposition}

  \begin{proof} Clear, since $V$ is the unique non-split extension up to scaling, and we may choose our scale factor so that $t^{1-k} \bar f$ corresponds to $v_{\alpha^{-1}} + v_{\beta^{-1}}$.
  \end{proof}

  \begin{corollary}\label{corr:lattices}
   Suppose that $V_f^*$ is non-split. Let $T$ be the lattice of Lemma \ref{lemma:lattices}. Then $T_f^* \supseteq T$, with equality if and only if $T_f^* / \pe T_f^*$ is non-split as a mod $p$ representation of $G_{\Qp}$.
  \end{corollary}

  \begin{proof} 
   As in \cite[\S 14.22]{kato04}, the element $\bar{f}$ is a basis for $\Fil^1 \Dcris(T_{\bar f})$. Twisting by $t^{1-k}$, we see that $T_f^*$ corresponds to a strongly divisible lattice in $\Dcris(V)$ whose intersection with $\Fil^0 \Dcris(V)$ is the $\calO_E$-span of $v_{\alpha^{-1}} + v_{\beta^{-1}}$. Applying Corollary \ref{cor:latticecd}, $T_f^*$ must be $T_{c, 0}$ for some $c \le 0$; this is residually non-split if and only if $c = 0$.
  \end{proof}

  Note that the local splitness of $T_f^* / \pe T_f^*$ at $p$ can be explicitly checked in certain cases, using \cite[Proposition 6.9]{pollackstevens09}.

  If $V_f^*$ is split, then we have an isomorphism $V_f^* \cong V_1 \oplus V_2$, where $V_1 = E(\chi^{k-1} \lambda(\varepsilon^{-1} \alpha))$ and $V_2 = E(\lambda(\alpha^{-1}))$ as above. If $n_1$ and $n_2$ are the natural basis vectors of the Wach modules of $V_1$ and $V_2$, then it is clear that $v_{\beta^{-1}} = \left(\tfrac{t}{\pi}\right)^{1-k} n_1$ and $v_{\alpha^{-1}} = n_2$ are a basis of eigenvectors of $\Dcris(V_1 \oplus V_2)$, and the nontrivial filtration step is $\Dcris(V_2) = E v_{\alpha^{-1}}$. Hence we may choose an isomorphism $V_f^* \to V_1 \oplus V_2$ such that $t^{1-k} \bar f$ maps to $v_{\alpha^{-1}}$. Since $V_1$ and $V_2$ have distinct mod $p$ reductions, the only possible lattices are direct sums of lattices in the factors, and hence we may assume that the above isomorphism maps $T_f^*$ to the lattice corresponding to the span of $n_1$ and $n_2$. However, this isomorphism is still not canonical, since it is only determined up to multiplying $n_1$ by an arbitrary element of $\calO_E^\times$.


 \section{Consequences for the \texorpdfstring{$L$}{L}-functions}

  We now use the results we have collected on the local representation $V_{f}^*|_{G_{\Qp}}$ to describe the $p$-adic $L$-functions of $f$. We recall the setup from \cite[\S 3.6]{leiloefflerzerbes10}. Since $V_f^*$ has non-negative Hodge-Tate weights $0$ and $k-1$, and the unique 1-dimensional quotient of $V_f^*$ is not the trivial representation, the theorem quoted in \S \ref{sect:iwasawa} shows that 
  \[ \Hiw(\Qp, V_f^*) \cong \NN(V_f^*)^{\psi = 1}.\]
  We let $\uCol$ denote the composition of this isomorphism with the map $1 - \vp : \NN(V_f^*)^{\psi = 1} \to \left(\vp^* \NN(V_f^*)\right)^{\psi = 0}$. Note that since the Hodge-Tate weights of $V_f^*$ are $\ge 0$, we have
  \[  \left(\vp^* \NN(V_f^*)\right)^{\psi = 0} \subseteq (\Brig)^{\psi = 0} \otimes_{\Qp} \Dcris(V_f^*).\]

  Let $\mathbf{z} \in \Hiw(\Qp, V_f^*)$. We thus have an element $\uCol(\mathbf{z}) \in \left(\vp^* \NN(V_f^*)\right)^{\psi = 0}$. We consider it as an element of $(\Brig)^{\psi = 0} \otimes_{\Qp} \Dcris(V)$ via the above inclusion, and define $L_\alpha(\mathbf{z})$ and $L_\beta(\mathbf{z})$ to be its projections to the eigenspaces, so
  \[ \uCol(\mathbf{z}) = L_\alpha(\mathbf{z}) v_{\alpha^{-1}} + L_\beta(\mathbf{z}) v_{\beta^{-1}}.\]
  We can consider $L_\alpha(\mathbf{z})$ and $L_\beta(\mathbf{z})$ as power series in $\pi$ lying in $(\Brig)^{\psi = 0}$. Alternatively, we may regard them as distributions via the Mellin transform isomorphism 
  \[ \mathfrak{M}: \calH(\Gamma) \rTo^\cong (\Brig)^{\psi = 0};\]
  which maps a group element $\gamma$ to $\gamma \cdot (1 + \pi)$. Perrin-Riou's theory implies that $L_\alpha(\mathbf{z})$ is a distribution of order 0, while $L_\beta(\mathbf{z})$ has order $k-1$.

  If $\omega$ is a Dirichlet character (to any modulus), let $L(f, \omega, s)$ denote the complex $L$-function of $f$ twisted by $\omega$, $\sum a_n(f) \omega(n) n^{-s}$. Recall (e.g.~from \cite[\S 3.1.2]{colmez04}) that there are nonzero complex numbers $\Omega_+$ and $\Omega_-$ such that 
  \begin{equation}\label{eq:Ltilde}
   \tilde L(f, \omega, j + 1) := \frac{\Gamma(j+1)}{(2 \pi i)^{j+1} \Omega_{\pm}} L(f, \omega, j+1) \in \QQ(f, \omega)
  \end{equation}
  where $\Omega_{\pm}$ denotes $\Omega_+$ if $(-1)^{j+1} \chi(-1) = 1$ and $\Omega_-$ if $(-1)^{j+1} \chi(-1) = -1$. Here $\QQ(f, \omega)$ denotes the finite extension of $\QQ$ generated by the values of $\omega$ and the coefficients $a_n(f)$.

  \begin{theorem}[{\cite[Theorem 16.6]{kato04}}]
   There exists an element $\kato \in  \Hiw(\Qp, V_f^*)$ such that for any finite-order character $\omega$ of $\Gamma$ of conductor $p^n$, and any $0 \le j \le k-2$,
   \[ L_{\alpha}(\kato)(\chi^j \omega) = 
    \begin{cases}
     \left(1 - p^j \alpha^{-1}\right)\left(1 - \varepsilon p^{k-2-j}\alpha^{-1} \right) \tilde L(f, 1, j + 1) &\text{if $n = 0$,}\\
     \alpha^{-n} p^{n(j+1)}\displaystyle\frac{\tilde L(f, \omega^{-1}, j + 1)}{G(\omega^{-1})} &\text{if $n \ge 1$,}
    \end{cases}
   \]
    where $G(\omega^{-1})$ is the Gauss sum. 
  \end{theorem}

  Hence $L_\alpha(\kato)$ is a distribution of order 0 on $\Gamma$ whose values at special characters are given by \eqref{eq:interpolating}; so it is equal to the $p$-adic $L$-function $L_{p, \alpha}$. We define $L_{p, \beta} = L_\beta(\kato)$; if $V_{f}^*$ is non-split, this satisfies \eqref{eq:interpolating} for the root $\beta$, while if $V_f^*$ is split, $L_{p, \beta}$ vanishes at all special characters. We now use the fact that $\uCol$ factors through $\left(\vp^* \NN(V_f^*)\right)^{\psi = 0}$, and the basis of the latter space given by proposition \ref{prop:basis}, to give a decomposition of these two distributions.

  \begin{definition}
   Let $n_1, n_2$ be the basis of $\NN(T_f^*)$ defined above. We let $L_{p, 1}$ and $L_{p, 2}$ be the unique elements of $\Lambda_{E}(\Gamma)$ such that
  \[ \uCol(\kato) = L_{p, 1} \cdot (1 + \pi)\vp(n_1) + L_{p, 2} \cdot (1 + \pi)\vp(n_2).\]
  \end{definition}
 
  \begin{proposition}\label{prop:integralLfunctions}
   If the image of $\Gal(\overline{\QQ} / \QQ_\infty)$ in $\GL(T_f^*)$ contains a conjugate of $\operatorname{SL}_2(\Zp)$, then $L_{p, 2}$ lies in $\Lambda_{\calO_E}(\Gamma)$. If in addition $V_f^*$ is split at $p$, or it is non-split and the residual representation $T_f^* / \pe T_f^*$ is also non-split, then the same holds for $L_{p, 1}$.
  \end{proposition}

  \begin{proof}
   If the hypothesis on the image of the global representation is satisfied, then $\kato \in \Hiw(T_f^*)$, by Theorem 12.5(4) and Theorem 12.6 of \cite{kato04}. Since the image of $\NN(T_f^*)$ in the quotient is the $\EA$-span of $n_2$, this implies that $L_{p, 2}$ is integral. The additional assumptions imply the stronger statement that $\NN(T_f^*)$ is the span of $n_1$ and $n_2$ (Corollary \ref{corr:lattices}), so we also obtain integrality for $L_{p, 1}$.
  \end{proof}

  Using the formulae of the previous section relating $v_{\alpha^{-1}}$ and $v_{\beta^{-1}}$ to $n_1$ and $n_2$, we can write the functions $L_{p,\alpha}$ and $L_{p,\beta}$ in terms of the $L_{p, i}$. 

  \begin{theorem}\label{prop:relations-mult} The following relations hold in $E \otimes \calH(\Gamma)$:
   \begin{enumerate}
    \item[(a)] If $V_f^*$ is locally split, then
    \[
     \begin{cases}
      \alpha L_{p, \alpha} &= L_{p, 2}\\
      \beta L_{p, \beta} &= L_{p,1} \mathfrak{M}^{-1}\left( (1 + \pi)\vp\left(\frac{t}{\pi}\right)^{k-1} \right)
     \end{cases}
    \]
    \item[(b)] If $V_f^*$ is not locally split, then
    \[
     \begin{cases} 
      \alpha L_{p, \alpha} &= L_{p, 2}\\
      -a(0) \beta L_{p, \beta} &= L_{p,1} \mathfrak{M}^{-1}\left( (1 + \pi)\vp\left(\frac{t}{\pi}\right)^{k-1} \right) - L_{p, 2}\mathfrak{M}^{-1}\left( (1 + \pi)\vp(a) \right)
     \end{cases}
    \]
   \end{enumerate}
  \end{theorem}

  \begin{proof} 
   In the non-split case, we use Proposition \ref{prop:evects} to write
   \begin{align*}
    n_1 &= -\tfrac{1}{a(0)}\left(\tfrac{t}{\pi}\right)^{k-1} v_{\beta^{-1}},\\
    n_2 &= v_{\alpha^{-1}} + \frac{a}{a(0)} v_{\beta^{-1}}.
   \end{align*}   

   Substituting these into the identity 
   \[ L_{p, \alpha} \cdot (1 + \pi) v_{\alpha^{-1}} + L_{p, \beta} \cdot (1 + \pi) v_{\beta^{-1}} =  L_{p, 1} \cdot (1 + \pi) \vp(n_1) + L_{p, 2}\cdot (1 + \pi) \vp(n_2),\]
   where the $\cdot$ denotes the action of $\calH(\Gamma)$ on $\left(\Brig\right)^{\psi = 0} \otimes_{\Qp} \Dcris(V)$, we obtain
   \begin{align*}
    & L_{p, \alpha} \cdot (1 + \pi) v_{\alpha^{-1}} + L_{p, \beta} \cdot (1 + \pi) v_{\beta^{-1}} \\
     & =  L_{p, 1} \cdot \frac{-1}{a(0)\beta} (1 + \pi)  \vp \left(\tfrac{t}{\pi}\right)^{k-1} v_{\beta^{-1}} 
     + L_{p, 2} \cdot (1 + \pi)\left(\frac{1}{\alpha} v_{\alpha^{-1}} + \frac{\vp(a)}{a(0) \beta}v_{\beta^{-1}}\right).
   \end{align*}
   Since $(1 + \pi)v_{\alpha^{-1}}$ and $(1 + \pi)v_{\beta^{-1}}$ are clearly a basis for $(\Brig)^{\psi = 0} \otimes \Dcris(V)$ as a $\calH(\Gamma)$-module, we can project onto each of these to obtain the proposition.

   In the split case, one argues identically using the formulae $n_1 = \left(\tfrac{t}{\pi}\right)^{k-1} v_{\beta^{-1}}$ and $n_2 = v_{\alpha^{-1}}$.
  \end{proof}

  Note that $\vp\left(\tfrac{t}{\pi}\right)^{k-1}$ has a zero of order $k-1$ at $\zeta - 1$, for any root of unity $\zeta$ of order $p^n$, $n \ge 2$. It is straightforward to see (using Theorem 5.4 and Lemma 5.9 of \cite{leiloefflerzerbes10}) that this is equivalent to $\mathfrak{M}^{-1}((1 + \pi)\vp\left(\tfrac{t}{\pi}\right)^{k-1})$ vanishing at every character of $\Gamma$ of the form $\chi^j \omega$, where $\chi$ is the cyclotomic character, $0 \le j \le k-2$ and $\omega$ is a finite-order character not factoring through $\Delta$. Hence the factor multiplying $L_{p, 1}$ in Proposition \ref{prop:relations-mult} vanishes at all but finitely many of the points corresponding to critical values of the complex $L$-function.

  In fact $L_{p, 1}$ vanishes at most of the remaining points: 

  \begin{proposition} 
   If $V_f^*$ is locally split, then the distribution $L_{p, 1}$ vanishes at $z \mapsto z^i \eta(z)$, for any $0 \le i \le k-2$ and any character $\eta$ of $\Zp^\times$ factoring through $\Delta$. If $V_f^*$ is not locally split, this is true at all characters of this form with $\eta$ nontrivial.
  \end{proposition}

  \begin{proof}
   If the representation is split, $L_{\beta}$ is known to vanish at all special characters. Since we have $\beta L_{p, \beta} = L_{p, 1} \mathfrak{M}^{-1}((1 + \pi)\vp\left(\tfrac{t}{\pi}\right)^{k-1})$, and the second factor on the right-hand side does not vanish at the characters $z \mapsto z^i \eta(z)$, $L_{p, 1}$ must do so.

   In the non-split case, we suppose that $\eta$ is nontrivial. Then we must have
   \[ \beta L_{p, \beta}(\chi^i \eta) = \alpha L_{p, \alpha}(\chi^i \eta),\]
   since both sides are equal to $p^{j+1} \tilde L(f, \eta^{-1}, 1 + j) / G(\eta^{-1})$. Substituting the formulae of Proposition \ref{prop:relations-mult}, we obtain
   \begin{multline*}
    L_{p, 2}(\chi^j \eta) \mathfrak{M}^{-1}\Big( (1 + \pi)\vp\big(a - a(0)\big)\Big)(\chi^j \eta) \\
    = L_{p, 1}(\chi^j \eta) \mathfrak{M}^{-1}\left( (1 + \pi)\vp\left(\tfrac{t}{\pi}\right)^{k-1}\right).
   \end{multline*}
   By construction $b := a - a(0)$ vanishes to order $k$ at 0, so $(1 + \pi)\vp(b)$ vanishes to order $k$ at $\zeta_p - 1$ for any nontrivial $p$-th root of unity $\zeta$. Hence the distributions $\partial^j (1 + \pi) \vp(b)$ vanish at $\zeta_p - 1$, for $i = 0, \dots, k-2$, where $\partial = (1 + \pi) \tfrac{\mathrm{d}}{\mathrm{d}\pi}$; equivalently, $(1 + \pi) \vp(b)$ pairs to zero with any function on $\Zp$ whose restriction to each coset of $p\Zp$ is a polynomial of degree $\le k-2$. In particular, it pairs to zero with the characters $z\mapsto z^j \eta(z)$ (extended to functions on $\Zp$ zero on $p\Zp$). Since $\mathfrak{M}^{-1}\left( (1 + \pi)\vp\left(\tfrac{t}{\pi}\right)^{k-1}\right)$ does not vanish at these characters, $L_{p, 1}$ must vanish.
  \end{proof}

  This proposition, together with the preceding discussion, imply that the distribution $L_{p, 1} \mathfrak{M}^{-1}((1 + \pi)\vp\left(\tfrac{t}{\pi}\right)^{k-1})$ vanishes at every locally algebraic character of degree $\le k-2$ that is not algebraic, and at every locally algebraic character in the split case. In the split case, this distribution is simply $L_{p, \beta}$, and one deduces the well-known fact that $L_{p, \beta}$ ``pretends rather convincingly to be 0'' (see \cite[Remarque 4.12]{colmez04}).

  \begin{remark}
   From the formula $a = t^{k-1} \alpha x_k + p^{1-k} \mu \vp(a)$, we deduce that 
   \[ a(\zeta_{p^j} - 1) = p^{1-k} \mu a(\zeta_{p^{j}}^p - 1)\]
   for any $j \ge 1$, where $\zeta_{p^j}$ is a $p^j$-th root of unity. Since $p^{1-k} \mu = \alpha / \beta$, this gives
   \[ a(\zeta_{p^j} - 1) = \left(\frac{\alpha}{\beta}\right)^j a(0).\]
   Since $a - a(0)$ vanishes to order $\ge k-1$ at 0, $a - \left(\frac{\alpha}{\beta}\right)^j a(0)$ vanishes to order $k-1$ at $\zeta_{p^j} - 1$.
   This gives a purely analytic proof that for any $F \in \calH(\Gamma)$ satisfying the interpolation property of \eqref{eq:interpolating} for the critical-slope root $\beta$, the distribution
   \[ G = \alpha L_{p, \alpha} \mathfrak{M}^{-1}\left( (1 + \pi) \vp(a) \right) - a(0) \beta F\]
   vanishes at all special characters of $\Gamma$ of conductor $> 1$. Hence $G$ factorises as $\mathfrak{M}^{-1}\left( (1 + \pi)\vp\left(\frac{t}{\pi}\right)^{k-1}  \right) H$ for some distribution $H$, and if $F$ has order $k-1$, then $H$ must be in $\Lambda_E(\Gamma)$. 
   
   In particular, taking $F$ to be the analytic critical-slope $L$-function $L_{p, \beta}^{\PS}$, we obtain a decomposition of $L_{p, \beta}^{\PS}$ analogous to Proposition \ref{prop:relations-mult}. However, without the above interpretation of $H$ via Wach modules, it is not clear how one could determine whether or not $H$ was integral.
  \end{remark}

  In the non-split case, we also obtain a formula for $L_{p, 1}(\chi^i)$, for $i = 0, \dots, k-2$, which allows us to show that it is non-vanishing in some cases:

  \begin{proposition}
   If $V_f^*$ is non-split and at least one of the $L$-values $L(f, j)_{j = 1, \dots, k-1}$ is non-zero, then $L_{p, 1} \ne 0$.
  \end{proposition}

  \begin{proof}
   For $0 \le j \le k-2$, we have
   \[ \mathfrak{M}^{-1}\left( (1 + \pi) \vp(a) \right)(\chi^j) = a(0),\]
   since $a - a(0)$ vanishes to degree $\ge k-1$ at the origin. Thus 
   \[ L_{p, 1}(\chi^j) \mathfrak{M}^{-1}\left( ( 1+ \pi)\vp\left(\tfrac{t}{\pi}\right)^{k-1}\right)(\chi^j) = a(0) \left(\alpha L_{p, \alpha} - \beta L_{p, \beta}\right).\]

   On the right-hand side, 
   \[ a(0) = \frac{(k-2)! \alpha}{1 - p^{1-k} \mu} = \frac{(k-2)!\alpha}{1 - \alpha/\beta} = \frac{(k-2)! \alpha \beta}{\beta - \alpha},\]
   and substituting the values of $L_{p,\alpha}$ and $L_{p, \beta}$ at $\chi^j$ from Equation \eqref{eq:interpolating} and simplifiying, we obtain (eventually)
   \begin{align*}
    L_{p, 1}&(\chi^j) \cdot \mathfrak{M}^{-1}\left( (1 + \pi)\vp\left(\tfrac{t}{\pi}\right)^{k-1} \right)(\chi^j) \\
    &= -(k-2)! (p-1) p^{k-2} \varepsilon \tilde L(f, 1, 1 + j).
   \end{align*}
  \end{proof}

  If $k \ge 3$, this is sufficient to show that $L_{p, 1} \ne 0$, since the complex $L$-function $L(f, j)$ does not vanish for $j > \frac{k}{2}$. If $k = 2$, the only character at which we can relate the value of $L_{p, 1}$ to the complex $L$-function is the trivial character, so when $L(f, 1) = 0$ (which can of course happen) we cannot show that $L_{p, 1} \ne 0$.

 \section{Newton polygons and Mellin transforms}\label{newton}

  In this section, we take $k = 2$, and present some explicit consequences of the above analysis for the algebraic critical-slope $L$-function $L_{p, \beta}$. For $s > 0$ let $C_s$ denote the closed affinoid disc $\{ X : |X| \le p^{-s}\}$. (For our purposes it will suffice to take $s$ rational; if $s$ is irrational, this space is not defined as an affinoid space, but it can be interpreted as a Berkovich space). For a rigid-analytic function $f$ on the open unit disc, we write $v_s(f) = \inf_{x \in C_s} \mathrm{ord}_p(f(x))$; note that $v_s(f)$ is clearly an increasing function of $s$.

  \begin{proposition}
   The function $s \mapsto v_s(f)$ is continuous, piecewise-linear, and concave. For any $s \ge 0$, the left-hand derivative of $v_s(f)$ at $s$ is the number of zeroes of $f$ on $C_s$ (counted with multiplicity), and the right-hand derivative is the number of zeros of $f$ on the open disc $\{ X : |X| < p^{-s}\}$.
  \end{proposition}

  \begin{proof} 
   This is simply a restatement of the standard theory of the Newton polygon.
  \end{proof}

  Let $\gamma_1$ be a generator of $\Gamma_1$, and let $x = \gamma_1 - 1$, so $\calH(\Gamma_1)$ is the ring of power series in $x$ converging on the unit disc.

  \begin{proposition}
   Let $f \in \Brig$ and let $g = \mathfrak{M}^{-1}\left((1 + \pi) \vp(f)\right) \in \calH(\Gamma_1)$. Then for any $s$ with $0 < s < 1$, we have $v_s(f) = v_s(g)$.
  \end{proposition}

  \begin{proof}
   Let us suppose $f = \sum a_n \pi^n$. Then $g = \sum_{n \ge 0} a_n \phi_n(x)$, where $\phi_n(x) = \mathfrak{M}^{-1}\left( (1 + \pi) \vp(\pi)^n\right)$.

   We know that $\phi_n(x) = \tau_n(x) x (x - \lambda_1) \dots (x - \lambda_{n-1})$, where $\lambda_i = \chi(\gamma_1)^i - 1 \in p\Zp$ and $\tau_n \in \Lambda_{\Qp}(\Gamma_1)^\times$. Since $(1 + \pi)\vp(\pi)^n \in \mathbb{A}^+_{\Qp} \setminus p \mathbb{A}^+_{\Qp}$, we must have $\tau_n \in \Lambda_{\Zp}(\Gamma_1)^\times$.

   Let us write $d_n = \tau_n(0)$. Then we have
   \begin{align*}
    v_s \left(\phi_n(x) - d_n x^n\right) = & v_s \Big( x^n \left(\tau_n(x) - d_n\right)\\
      & + \tau_n(x) \left( x(x-\lambda_1)\dots(x-\lambda_{n-1}) - x^n\right)\Big).
   \end{align*}
   We evidently have $x^n \left(\tau_n(x) - d_n\right) \in x^{n+1} \Zp[[x]]$, so 
   \[ v_s\left(x^n \left(\tau_n(x) - d_n\right)\right) \ge (n+1)s.\]
   For the second group of terms, the coefficient of $x^{n-j}$ in the product $x(x-\lambda_1)\dots(x-\lambda_{n-1}) - x^n$ is clearly divisible by $p^j$; since $\tau_n(x) \in \Zp^\times$, we have
   \begin{align*} 
    v_s\Big( \tau_n(x) \left( x(x-\lambda_1)\dots(x-\lambda_{n-1}) - x^n\right)\Big) & \ge \inf_{1 \le j \le n} (n-j)s + j \\
    & = (n-1)s + 1.
   \end{align*}
   Since $0 < s < 1$, both $(n + 1)s$ and $(n - 1)s + 1$ are strictly bigger than $v_s( d_n x^n) = ns$. Thus, in particular, $v_s( \phi_n(x)) = ns$.

   We now write
   \[ g = \sum_{n \ge 0} a_n \phi_n(x) = \left(\sum_{n \ge 0} a_n d_n x^n\right) + \left(\sum_{n \ge 0} a_n \left(\phi_n(x) - d_n x^n\right)\right).\]
   Clearly we have
   \[ v_s\left(\sum_{n \ge 0} a_n d_n x^n\right) = \inf_{n \ge 0} (ns + \mathrm{ord}_p a_n) = v_s(f).\]
   On the other hand, 
   \begin{align*}
     v_s\left(\sum_{n \ge 0} a_n\left(\phi_n(x) - d_n x^n\right)\right) & \ge \inf_{n \ge 0} \Big( \inf\big((n-1)s + 1, (n + 1)s\big) + \mathrm{ord}_p a_n\Big) \\
     & = v_s(f) + \inf(s, 1 - s).
   \end{align*}
   Hence we must have $v_s(g) = v_s(f)$.
  \end{proof}

  Combining the two preceding propositions, we see that the zeros of the power series $f$ and $g$ lying ``near the boundary'' must have the same valuations; the zeros inside the closed disc $|X| \le p^{-1}$ are equal in number, but can be in very different places within this disc, as the examples $f = \pi^n$ show.

  \begin{corollary}\label{corr:bounds-a}
   Let $\mu \in \calO_E^\times$, $\mu \ne 1$, and let $f$ be the unique element of $E \otimes \Brig$ such that
   \[ \left(1 - \frac{\mu}{p} \vp\right)(f) = \frac{t}{\pi} + \frac{t}{2}.\]
   Then $f$ has $(p-1)$ zeros of valuation $\ge \frac{1}{p-1}$, $p^i(p-1)^2$ zeros of valuation $\tfrac{1}{p^i(p-1)^2}$ for each integer $i \ge 0$, and no other zeros. Moreover, $v_s(f) < v_s\left(\tfrac t \pi\right)$ for $s < \frac{1}{(p-1)^2}$; and we have
   \begin{align*}
    \liminf_{s \to 0}\Big( v_s\left(\tfrac t \pi\right) - v_s\left(f\right)\Big) &= \frac{1}{(p-1)^2}\\
    \limsup_{s \to 0}\Big( v_s\left(\tfrac t \pi\right) - v_s\left(f\right)\Big) &= \frac{1}{(p-1)}.
   \end{align*}
  \end{corollary}

  (We have $f = a/\alpha$ in the notation of the previous sections, in the case $k = 2$.)

  \begin{proof}
   Let us calculate $v_s(f)$. We suppose first that $s > \frac{1}{p-1}$. Then the formal series expansion
   \[ f = \sum_{\substack{n \ge 0 \\ n \ne 1}} \frac{B_n t^n}{n!(1 - p^{n-1} \mu)}\]
   of proposition \ref{prop:a-formula} is convergent, and the disc $|\pi| \le p^{-s}$ corresponds to $|t| \le p^{-s}$; hence
   \[ v_f(s) = \inf_{\substack{n \ge 0 \\ n \ne 1}} \ord_p \left( \frac{B_n}{n! (1 - p^{n-1} \mu)}\right) + ns.\]

   For $n \ge 2$, $(1 - p^{n-1} \mu) \in \calO_E^\times$, and hence $\ord_p\left( \frac{B_n}{n! (1 - p^{n-1} \mu)}\right) = \ord_p \left( \frac{B_n}{n!}\right)$. We have
   \begin{align*}
    \inf_{n \ge 2} \ord_p \left( \frac{B_n}{n!}\right) + ns &= \inf_{|t| \le p^{-s}} \ord_p\left( \frac{t}{e^t - 1} - 1 + t/2\right) \\
    &= \inf_{|\pi| \le p^{-s}} \ord_p \left( \frac{t}{\pi} - 1 + t/2\right)\\
    &= \inf_{n \ge 2} \ord_p \left(\frac{1}{n + 1} - \frac{1}{2n}\right) + ns.
   \end{align*}

   Clearly, if $n \ge p-1$ and $n'$ is the largest integer $\le n$ of the form $p^j - 1$, or $n' = 2$ if $n < p-1$, then we have
   \[ \ord_p \left(\frac{1}{n' + 1} - \frac{1}{2n'}\right) + n's < \ord_p \left(\frac{1}{n + 1} - \frac{1}{2n}\right) + ns.\]
   Thus the infimum is attained either at $n = 2$ or at $n = p^j - 1$ for some $j \ge 1$. We calculate that the term for $n = p^j - 1$ is $(p^j - 1)s - j$, which is a strictly increasing function of $j$ for any $s > \frac{1}{p-1}$. Hence the infimum is 
   \[ \inf \left( (p - 1)s - 1, 2s \right).\]
   (If $p = 3$, then the $2s$ term does not appear, but the other term is $2s-1$ which is smaller anyway, so the formula is true as stated.)

   If we include also the term in the original sum for $n = 0$,we deduce that for any $s$ in this range
   \[ v_s(f) = \inf\left( 1, 2s, (p - 1)s - 1\right).\]
   One checks that if $p \ge 5$, this gives
   \[ v_s(f) =
    \begin{cases}
     (p-1)s - 1 &\text{if $\tfrac{1}{p-1} \le s \le \tfrac{1}{p-3}$}\\
     2s &\text{if $\tfrac{1}{p-3} \le s \le \tfrac{1}{2}$}\\
     1 &\text{if $s \ge \tfrac{1}{2}$.}
    \end{cases}
   \]
   whereas if $p = 3$ we obtain
   \[ v_s(f) =
    \begin{cases} 2s - 1 &\text{if $\frac{1}{2} \le s \le 1$}\\
     1 &\text{if $s \ge 1$.}
    \end{cases}
   \]
   Hence the zeros of $f$ with valuation $> \frac{1}{p-1}$ are: two zeros of valuation $1$ if $p = 3$; four zeros of valuation $\frac{1}{2}$ if $p = 5$; and two of valuation $\frac{1}{2}$ and $p-3$ of valuation $\frac{1}{p-3}$ if $p \ge 7$. In each case, the total number of zeros with valuations in this range is $p-1$.

   We now use this intensive study of $v_f(s)$ for relatively large $s$ to describe $v_f(s)$ for all smaller $s$. From the equation $f = t/\pi + t/2 + p^{-1} \mu \vp(f)$, we deduce that for $s$ in the interval $(\frac{1}{p-1}, \frac{p}{p-1})$ we have
   \begin{align*}
    v_{s/p}(f) &\ge \inf(v_{s/p}(t/\pi), -1 + v_s(f)) \\
    & = \inf(v_{s/p}(t) - s/p, -1 + v_s(f))\\
    &= -1 + \inf\left(\left(1 - \tfrac{1}{p}\right)s, v_s(f)\right).
   \end{align*}
   We find that if $\tfrac{1}{p-1} < s < \tfrac{p}{(p-1)^2}$, the term $-1 + v_s(f)$ is strictly smaller, whereas for larger $s$, the term $-1 + \left(1 - \tfrac{1}{p}\right)s$ is strictly smaller; hence this is the exact value of the left-hand side (by the ultrametric property), and by continuity this is the case at the crossover point $s = \tfrac{p}{(p-1)^2}$. This determines $v_s(f)$ in the in the interval $(\frac{1}{p(p-1)}, \frac{1}{p-1})$: we have $v_{\frac{1}{p(p-1)}}(f) = -1, v_\frac{1}{(p-1)^2}\left(f\right) = -1 + \frac{1}{p-1}$, and $v_\frac{1}{p-1}(f) = 0$, and $v_s(f)$ is linear between these points.

   We now consider the interval from $\frac{1}{p^{n+1}(p-1)}$ to $\frac{1}{p^{n}(p-1)}$, for $n \ge 1$. By iterating the functional equation, we find that
   \[ v_{s/p^i}(f) \ge -i + \inf\left(v_f(s), \left(1 - \frac{1}{p}\right)s, \left(1 - \frac{1}{p^2}\right)s, \dots, \left(1 - \frac{1}{p^i}\right)s\right).\]
   The terms $\left(1 - \frac{1}{p^2}\right)s, \dots, \left(1 - \frac{1}{p^i}\right)s$ are clearly strictly larger than $\left(1 - \frac{1}{p}\right)s$, so we deduce that for $s \ge \frac{1}{p-1}$, we have $v_f(s/p^i) = -i + \inf\left(v_f(s), \left(1 - \frac{1}{p}\right)s\right) = 1 - i + v_f(s/p)$. This gives the locations of the zeros in the statement of the proposition, and shows that $v_s(f) < v_s(\tfrac{t}{\pi})$ for all $s < \frac{1}{(p-1)^2}$.

   Finally, we establish the formulae for the limits inferior and superior. For $\frac{1}{p-1} \le s \le \frac{p}{p-1}$ and $i \ge 1$ we have 
   \begin{align*}
    v_{s/p^i}(t/\pi) - v_{s/p^i}(f) &= (v_{s/p^i}(t) - s/p^i) - (-i + \inf\left(\left(1 - \tfrac{1}{p}\right)s, v_s(f)\right) \\
    &= \left(1 - \frac{1}{p^i}\right)s - \inf\left(\left(1 - \tfrac{1}{p}\right)s, (p-1)s - 1\right).
   \end{align*}
   One checks that the minimum value of this expression is attained at $s = \frac{p}{(p-1)^2}$, where it is equal to $\frac{p^{i-1} - 1}{p^{i-1}(p-1)^2}$. Hence for any $s \le \frac{1}{p^i(p-1)}$, we have $v_s(t/\pi) - v_s(f) \ge \frac{p^{i-1} - 1}{p^{i-1}(p-1)^2}$ and equality occurs for $s = \frac{1}{p^i(p-1)^2}$; so the limit inferior as $s \to 0$ is $\frac{1}{(p-1)^2}$, as claimed. On the other hand, the maximum value is $\frac{p^{i-1} - 1}{p^{i-1}(p-1)}$, attained at both of the endpoints, so the limit superior is $\frac{1}{p-1}$.
  \end{proof}

  We now use this to describe the $L$-functions. We note that $\calH(\Gamma) = \bigoplus_{\eta} e_\eta \calH(\Gamma)$, where the sum is over the characters of $\Delta$ and $e_\eta$ is the corresponding idempotent. For $g \in \calH(\Gamma)$, we let $g^\eta$ be the unique element of $\calH(\Gamma_1)$ such that $e_\eta g^\eta = e_\eta g$. 

  For a character $\eta$ of $\Delta$, let $\lambda_1^\eta, \mu_1^\eta, \lambda_2^\eta, \mu_2^\eta$ be the Iwasawa $\lambda$- and $\mu$-invariants of $L_{p, 1}^\eta$ and $L_{p, 2}^\eta$. Note that $\lambda_2^\eta, \mu_2^\eta$ are equal to the corresponding invariants of the unit root $p$-adic $L$-function $L_{p, \alpha}$, which can be calculated in many cases; we know of no easy way to evaluate $\lambda_1^\eta, \mu_1^\eta$, but Proposition \ref{prop:integralLfunctions} gives conditions under which $\mu_2^\eta$ is forced to be non-negative.

  \begin{theorem}\label{maintheorem} 
   Let $\eta$ be a character of $\Delta$ and let $\lambda_1^\eta, \mu_1^\eta, \lambda_2^\eta, \mu_2^\eta$ be as above. Suppose that $V_f^*$ is non-split at $p$.
   \begin{enumerate}
    \item[(a)] If $\mu_2^\eta < \frac{1}{(p-1)^2} + \mu_1^\eta$, then for all sufficiently small $s$ we have
   \[ v_s(L_{p, \beta}^\eta) = \lambda_2^\eta s + \mu_2 + v_s(f),\]
   where $f$ is as above. In particular, for $n \gg 0$, $L_{p, \beta}^\eta$ has $p^n(p-1)^2$ zeros of valuation $\tfrac{1}{p^n(p-1)^2}$, and the total number of zeros of valuation $> r_n$ is $p^n(p-1) + \lambda_2^\eta$.
    \item[(b)] If $\mu_2^\eta > \frac{1}{p-1} + \mu_1^\eta$, then for sufficiently small $s$ the formula becomes
   \[ v_s(L_{p, \beta}^\eta) = \lambda_1^\eta s + \mu_1 + v_s(\tfrac{t}{\pi}),\]
   so for $n \gg 0$ there are $p^n(p-1)$ zeros of valuation $\tfrac{1}{p^n(p-1)^2}$ and the number of zeros of valuation $> r_n$ is $p^n - 1 + \lambda_1^\eta$.
   \end{enumerate}
  \end{theorem}

  \begin{proof}
   We have a decomposition
   \[ -a(0) \beta L_{p, \beta} = L_{p, 1} \mathfrak{M}^{-1}\left( (1 + \pi) \vp\left(\tfrac{t}{\pi}\right)\right) - L_{p, 2} \mathfrak{M}^{-1}\big( (1+\pi) \vp(a)\big).\]
   For all sufficiently small $s$, we have $v_s(L_{p, i}^{\eta}) = \mu_i^{\eta} + s \lambda_i^\eta$. Hence, using the bounds of Corollary \ref{corr:bounds-a}, the hypotheses of case (a) force the second term to dominate, giving the stated formula for $v_s(L_{p, \beta}^\eta)$. Similarly in case (b), the first term must dominate for all sufficiently small $s$, giving the formula stated.
  \end{proof}

  This gives a (conditional) explanation of the phenomena observed in the examples of \cite[\S 9]{pollackstevens08} for the critical slope 3-adic and 5-adic $L$-functions of quadratic twists of the elliptic curves $X_0(11)$ and $X_0(14)$, with $\eta$ the trivial character of $\Delta$. In all of these cases the residual representations are locally non-split at $p$.

  The 3-adic representation of $X_0(11)$ (and hence of any of its twists) is surjective\footnote{This follows from the fact that its mod $3$ representation is surjective, as can be explicitly seen by calculating the Galois group of the 3-torsion field; and its $j$-invariant does not lie in the image of the rational function $f(x)$ of \cite{elkies06}. So its mod 9 representation is surjective, which implies that its 3-adic representation is also surjective.}. Hence the Kato zeta element for this curve is integral, and any twist such that $\mu_2^\eta = 0$ must satisfy the hypotheses of (a) above.

  In the case of 3-adic $L$-functions of twists of $X_0(14)$, the residual representation is globally reducible (but still locally non-split). Thus we cannot show that $L_{p, 1}$ is integral; but if we assume that this is the case, then again any twist of $X_0(14)$ with $\mu_2^\eta = 0$ satisfies the hypotheses of (a), which is consistent with the numerical results of \emph{op.cit.}.

  In the case of the 5-adic $L$-functions of twists of $X_0(11)$, let us also assume that $L_{p, 1}$ is integral. From this assumption, it follows that for twists by even quadratic characters, where $\mu_2^\eta = 0$, we obtain the pattern of zeros of part (a) of the theorem; but for twists by odd quadratic characters, where $\mu_2^\eta > 0$, we are in the situation of (b), unless $\mu_1^\eta$ is also positive.

\section*{Acknowledgements} This paper was written while the second author was visiting the University of Warwick; she thanks the number theory group for the hospitality.
\providecommand{\bysame}{\leavevmode\hbox to3em{\hrulefill}\thinspace}
\providecommand{\MR}[1]{\relax}
\renewcommand{\MR}[1]{%
 MR \href{http://www.ams.org/mathscinet-getitem?mr=#1}{#1}.
}
\providecommand{\href}[2]{#2}
\newcommand{\articlehref}[2]{\href{#1}{#2}}

\end{document}